\documentclass[12 pt]{amsart}
\usepackage{amscd,amsfonts,amssymb,amsmath}
\usepackage[utf8]{inputenc}
\usepackage{color}
\newtheorem{theorem}{Theorem}[section]
\theoremstyle{definition}

\newtheorem{corollary}[theorem]{Corollary}
\newtheorem{lemma}[theorem]{Lemma}
\newtheorem{proposition}[theorem]{Proposition}

\newtheorem{example}[theorem]{Example}
\newtheorem{remark}[theorem]{Remark}
\numberwithin{equation}{subsection}
\usepackage{graphicx}

\begin{document}
\title[Twisted virtual braid group]{Twisted virtual braid group}
\author[V. Bardakov]{Valeriy G. Bardakov}
\author[T. Kozlovskaya]{Tatyana A. Kozlovskaya}
\author[K. Negi]{Komal Negi}
\author[M. Prabhakar]{Madeti Prabhakar}

\address{Sobolev Institute of Mathematics, 4 Acad. Koptyug avenue, 630090, Novosibirsk, Russia.}
\address{Novosibirsk State Agrarian University, Dobrolyubova street, 160, Novosibirsk, 630039, Russia.}
\address{Regional Scientific and Educational Mathematical Center of Tomsk State University,~36 Lenin Ave., Tomsk, Russia.}
\email{bardakov@math.nsc.ru}

\address{Regional Scientific and Educational Mathematical Center of Tomsk State University,~36 Lenin Ave., Tomsk, Russia.}
\email{t.kozlovskaya@math.tsu.ru}

\address{Indian Institute of Technology, Ropar-140001, Punjab, India}
\email{komal.20maz0004@iitrpr.ac.in}

\address{Indian Institute of Technology, Ropar-140001, Punjab, India}
\email{prabhakar@iitrpr.ac.in}

\subjclass[2010]{ 20E07, 20F36, 57K12}
\keywords{Braid group,  pure braid group,  quandle, link invariant.}

\begin{abstract}
In this paper we study some subgroups and their decompositions in semi-direct product   of the twisted virtual braid group $TVB_n$. 
In particular, the twisted virtual pure braid group  $TVP_n$ is the kernel of an epimorphism of  $TVB_n$ onto the symmetric group $S_n$.
We find the set of generators and defining relations for $TVP_n$ and show that  $TVB_n = TVP_n \rtimes S_n$.
Further we prove that $TVP_n$ is a semi-direct product of some subgroup and abelian group $\mathbb{Z}_2^n$. As corollary we get that the virtual pure braid group $VP_n$ is a subgroup of  $TVP_n$.
Also, we  construct some other epimorphism of  $TVB_n$ onto $S_n$. Its kernel, $TVH_n$  is an analogous of $TVP_n$. We find its set of generators  and defining relations and construct its decomposition in a semi-direct product.

\end{abstract}

\maketitle

\section{Introduction}


M. O. Bourgoin~\cite{Bour} took a pioneering role in developing the theory of twisted knots, which is a generalization the theory of classical knots and virtual knots. He introduced  the twisted Jones polynomial and the twisted knot group for twisted knot.
Following this, Naoko Kamada~\cite{N} introduced polynomial invariants and specialized algebraic structures called quandles, designed for the analysis of twisted knots and links. Moreover, in~\cite{NS}, S.~Kamada and N.~Kamada explores the use of biquandles in the study of twisted knots.

For the studying of twisted links in \cite{NPK} was introduced the twisted virtual braid group $TVB_n$ and was proved that the closing a twisted virtual braid leads to a twisted link (analogous of Alexander's theorem), also it was some analogous of the Markov theorem, explains that a twisted virtual braid achieves uniqueness up to modulo certain transformations called Markov moves. 

The group $TVB_n$ is some analogous of the the braid group $B_n$.  During some last decades were introduced and intensively studied some generalizations of  $B_n$  (see, for example, \cite{B, Bir, KL}  and references therein). For example, virtual braid group $VB_n$, welded  braid group $WB_n$,
singular  braid group $SB_n$ and some other.
Any group of this type  has a pure subgroup, which is the kernel of an epimorphism onto $S_n$. 
The pure braid group $P_n$ is the kernel of the epimorphism from braid group $B_n$ to the symmetric group $S_n$. The virtual pure braid group $VP_n$ is the kernel of the epimorphism of the virtual braid group $VB_n$ onto $S_n$ that maps, for each $i$,  $\sigma_i$ and $\rho_i$ to the  $\rho_i$. Here  $\langle \rho_1, \rho_2, \ldots, \rho_{n-1} \rangle \cong S_n$.

The twisted virtual braid group $TVB_n$ (see \cite{NPK})
is generated by three families of elements,
$$
\sigma_i, ~~\rho_i,~~i=1, 2, \ldots,n-1,~~\gamma_j,~~j=1, 2, \ldots, n.
$$
One can define some epimorphisms of $TVB_n$ onto $S_n$. The first one is
$$
\varphi_P \colon TVB_n \to S_n,~~\sigma_i \mapsto \rho_i, ~~\rho_i \mapsto \rho_i,~~i=1, 2, \ldots,n-1,~~\gamma_j \mapsto e,~~j=1, 2, \ldots, n.
$$
Its kernel $\ker(\varphi_P)$ is the twisted virtual pure braid group $TVP_n$. 
The second epimorphism  is
$$
\varphi_H \colon TVB_n \to S_n,~~\sigma_i \mapsto e, ~~\rho_i \mapsto \rho_i,~~i=1, 2, \ldots,n-1,~~\gamma_j \mapsto e,~~j=1, 2, \ldots, n.
$$
Its kernel $\ker(\varphi_H)$ is denoted by $TVH_n$. This group is some analogous of a subgroup of $VB_n$, which was introduced in \cite{R} and studied in \cite{BB}.  
The analogous of $TVH_n$ for the singular braid group $SB_n$ was defined in \cite{GKM}, where it was found a presentation of this subgroup in the case  $n=3$.

In this paper we study some subgroups, their presentation and decomposition into semi-direct product   of the twisted virtual braid group $TVB_n$. In particular, we proved  that
$$
TVB_n = TVP_n \rtimes S_n = TVH_n \rtimes S_n.
$$ 
These decompositions show that  for the studying of $TVB_n$ it is need to study $TVP_n$ and~$TVH_n$.

The paper is organized as follows. 
In Section \ref{BD}, we recall definitions and some known facts  from classical and virtual braid theories, specifically concerning the pure braid group and pure virtual braid group.

In Section \ref{pure}, we delve into the group structures of $TVB_n, TVP_n, TVH_n$, and $TS_n$. 
Additionally, we investigate the decomposition of the groups $TVP_n$ and $TVH_n$, specifically, we establish that $TVP_n = PL_n \rtimes A_n$ and $TVH_n = HL_n \rtimes A_n$, where $A_n \cong \mathbb{Z}_2^n$. 
Furthermore, we demonstrate that the subgroups $TVP_n$ and $TVH_n$ are non-isomorphic for all $n\geq 3$.

In Section \ref{PT}, we define the  endomorphisms 
$$
\varphi_{PT} \colon TVB_n \to TS_n,~~\sigma_i \mapsto \rho_i, ~~\rho_i \mapsto \rho_i,~~i=1, 2, \ldots,n-1,~~\gamma_j \mapsto \gamma_j,~~j=1, 2, \ldots, n,
$$
$$
\varphi_{HT} \colon TVB_n \to TS_n,~~\sigma_i \mapsto e, ~~\rho_i \mapsto \rho_i,~~i=1, 2, \ldots,n-1,~~\gamma_j \mapsto \gamma_j,~~j=1, 2, \ldots, n,
$$
where 
$$
TS_n = \langle \rho_1, \rho_2, \ldots, \rho_{n-1}, \gamma_1, \gamma_2, \ldots, \gamma_n \rangle \leq TVB_n.
$$
and find a set of generators and defining relations for the kernel $\ker(\varphi_{PT})$ and $\ker(\varphi_{HT})$ which we will denote by $PT_n$ and $HT_n$,  respectively,  and defining relations for the  image $Im(\varphi_{PT})$ and $Im(\varphi_{HT})$ of these  endomorphisms.  
It can be inferred from this result that the subgroup $TS_n$ is isomorphic to $A_n \rtimes S_n$.
Also, we proved that 
$$TVB_n = PT_n \rtimes TS_n = HT_n \rtimes TS_n.$$

At the end of the paper  we formulate some open problems and suggest directions for further research.

\section*{Acknowledgments}
This work  is supported by the Ministry of Science and Higher Education of Russia (agreement  No.  075-02-2022-884). The third author would like to thank the University Grants Commission(UGC), India, for Research Fellowship with NTA Ref.No.191620008047. The fourth author acknowledges the support given by SERB research project(MATRICS) with F.No.MTR/2021/000394. This work is supported by the NBHM, Government of India under grant-in-aid with F.No.02011/2/20223NBHM(R.P.)/R\&D II/970.

\bigskip

\section{Basic definitions} \label{BD}

In this section we recall some known definitions which can be found in \cite{Artin, Bir1, Mar}. 

The braid group $B_n$, $n\geq 2$, on $n$ strands can be defined as
a group generated by $\sigma_1,\sigma_2,\ldots,\sigma_{n-1}$ with the defining relations
\begin{equation}
\sigma_i \, \sigma_{i+1} \, \sigma_i = \sigma_{i+1} \, \sigma_i \, \sigma_{i+1},~~~ i=1,2,\ldots,n-2, \label{eq1}
\end{equation}
\begin{equation}
\sigma_i \, \sigma_j = \sigma_j \, \sigma_i,~~~|i-j|\geq 2. \label{eq2}
\end{equation}
The geometric interpretation of  $\sigma_i$, its inverse $\sigma_{i}^{-1}$ and the unit $e$ of $B_n$ are depicted  in the Figure~\ref{figure1}.
\begin{figure}[ht]
\includegraphics[totalheight=6cm]{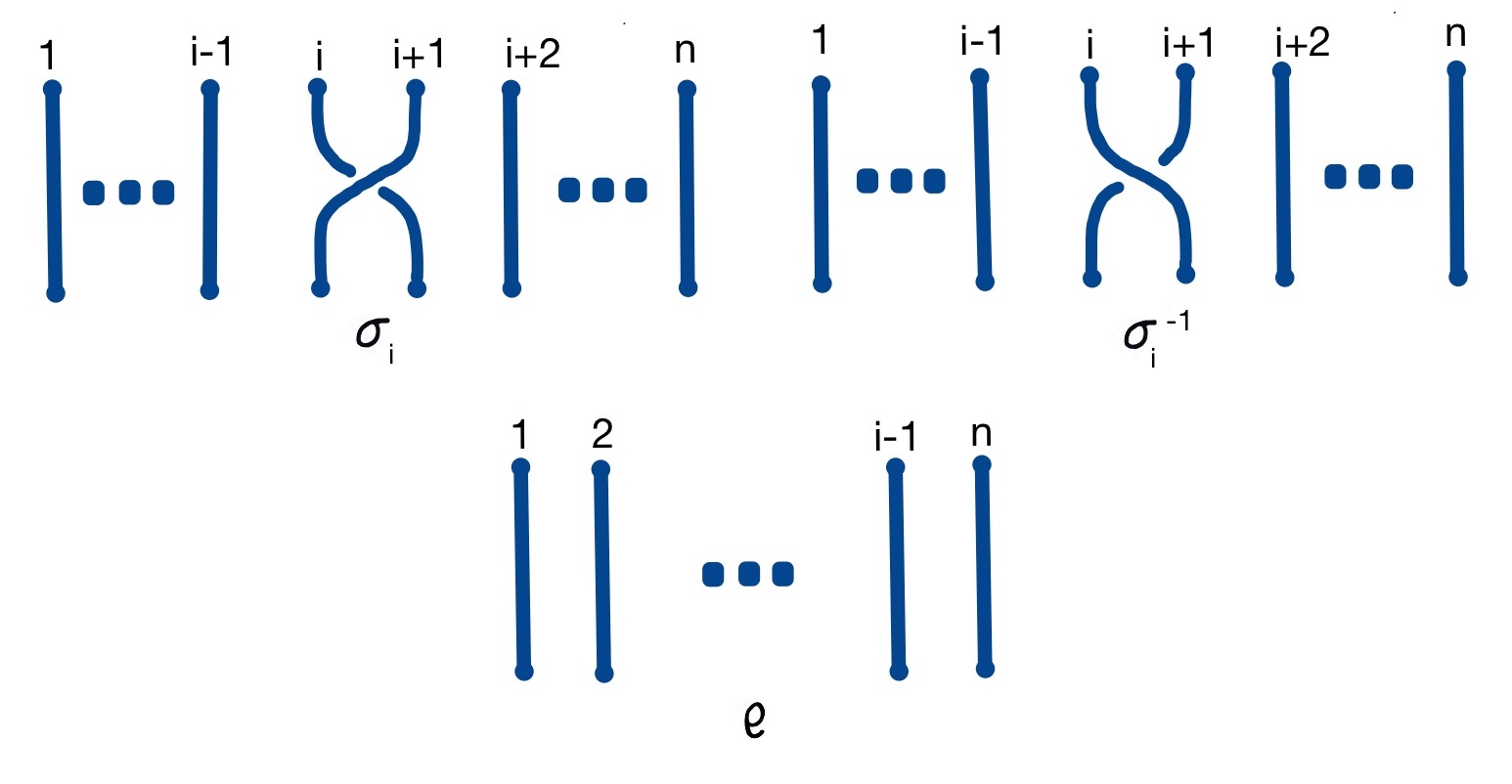}
\caption{The elementary braids $\sigma_i$, $\sigma_i^{-1}$ and the unit $e$} \label{figure1}
\end{figure}

There exists a homomorphism of $B_n$ onto the symmetric group $S_n$ on
$n$ symbols. This homomorphism  maps
 $\sigma_i$ to the transposition  $(i,i+1)$, $i=1,2,\ldots,n-1$.
The kernel of this homomorphism is called the
{\it pure braid group} and denoted by
$P_n$. The group $P_n$ is generated by elements $a_{ij}$, $1\leq i < j\leq n$.
These
elements can be expressed by the generators of
 $B_n$ as follows
$$
a_{i,i+1}=\sigma_i^2,
$$
$$
a_{ij} = \sigma_{j-1} \, \sigma_{j-2} \ldots \sigma_{i+1} \, \sigma_i^2 \, \sigma_{i+1}^{-1} \ldots
\sigma_{j-2}^{-1} \, \sigma_{j-1}^{-1},~~~i+1< j \leq n.
$$
In these generators $P_n$ is defined by relations
 \begin{align}
& a_{ik} a_{ij} a_{kj} = a_{kj} a_{ik} a_{ij},   \label{re2}\\
& a_{mj} a_{km} a_{kj} = a_{kj} a_{mj} a_{km},  ~\mbox{for}~m < j, \label{re3}\\
& (a_{km} a_{kj} a_{km}^{-1}) a_{im} = a_{im} (a_{km} a_{kj} a_{km}^{-1}),  ~\mbox{for}~i < k < m < j, \label{re4}\\
& a_{kj} a_{im} = a_{im} a_{kj},  ~\mbox{for}~k < i < m < j ~\mbox{or}~m < k. \label{re1}
\end{align}

Let $m_{kl} = \sigma_{k-1}  \, \sigma_{k-2} \ldots \sigma_l$ for $l < k$ and $m_{kl} = 1$
in other cases. Then the set
$$
\Lambda_n = \left\{ \prod\limits_{k=2}^n m_{k,j_k}~ |~ 1 \leq j_k
\leq k \right\}
$$
is a Schreier set of coset representatives of $P_n$ in $B_n$.

The subgroup $P_n$ is normal in $B_n$, and the quotient $B_n / P_n$ is  isomorphic to $S_n$. The generators of $B_n$ act on the generator $a_{ij} \in P_n$ by the rules:
 \begin{align}
& \sigma_k^{-1} a_{ij} \sigma_k =  a_{ij},  ~\mbox{for}~k \not= i-1, i, j-1, j, \label{c1}\\
& \sigma_{i}^{-1} a_{i,i+1} \sigma_{i} =  a_{i,i+1},   \label{c2}\\
& \sigma_{i-1}^{-1} a_{ij} \sigma_{i-1} =   a_{i-1,j},   \label{c3}\\
& \sigma_{i}^{-1} a_{ij} \sigma_{i} =  a_{i+1,j} [a_{i,i+1}^{-1}, a_{ij}^{-1}],  ~\mbox{for}~j \not= i+1 \label{c4}\\
& \sigma_{j-1}^{-1} a_{ij} \sigma_{j-1} =  a_{i,j-1},   \label{c5}\\
& \sigma_{j}^{-1} a_{ij} \sigma_{j} =  a_{ij} a_{i,j+1} a_{ij}^{-1},   \label{c6}
\end{align}
where $[a, b] = a^{-1} b^{-1} a b = a^{-1} a^b$.

Denote by
$$
U_{i} = \langle a_{1i}, a_{2i}, \ldots, a_{i-1,i} \rangle,~~~i = 2, \ldots, n,
$$
a subgroup of $P_n$.
It is known that $U_i$ is a free group of rank $i-1$. One can rewrite the defining relations of $P_n$ as the following conjugation rules (for $\varepsilon = \pm 1$):
  \begin{align}
& a_{ik}^{-\varepsilon} a_{kj}  a_{ik}^{\varepsilon} = (a_{ij} a_{kj})^{\varepsilon} a_{kj} (a_{ij} a_{kj})^{-\varepsilon},  \label{co1}\\
& a_{km}^{-\varepsilon} a_{kj}  a_{km}^{\varepsilon} = (a_{kj} a_{mj})^{\varepsilon} a_{kj} (a_{kj} a_{mj})^{-\varepsilon},  ~\mbox{for}~m < j, \label{co2}\\
& a_{im}^{-\varepsilon} a_{kj}  a_{im}^{\varepsilon} = [a_{ij}^{-\varepsilon}, a_{mj}^{-\varepsilon}]^{\varepsilon} a_{kj} [a_{ij}^{-\varepsilon}, a_{mj}^{-\varepsilon}]^{-\varepsilon},  ~\mbox{for}~i < k < m, \label{co3}\\
& a_{im}^{-\varepsilon} a_{kj} a_{im}^{\varepsilon} = a_{kj},  ~\mbox{for}~k < i < m < j ~\mbox{or}~  m < k. \label{co4}
\end{align}
The group $P_n$ is a semi--direct product of  the normal subgroup
$U_n$ and the group $P_{n-1}$. Similarly, $P_{n-1}$ is a semi--direct product of the free group
$U_{n-1}$  and the group $P_{n-2},$ and so on.
Therefore, $P_n$ is decomposable (see \cite{Mar}) into the following semi--direct product
$$
P_n=U_n\rtimes (U_{n-1}\rtimes (\ldots \rtimes
(U_3\rtimes U_2))\ldots),~~~U_i\cong F_{i-1}, ~~~i=2,3,\ldots,n.
$$

\subsection{Virtual braid group}
The {\it virtual braid group} $VB_n$  was introduced in \cite{Ka}. This group   is generated by the  braid group $B_n = \langle \sigma_1, \sigma_2, \ldots, \sigma_{n-1} \rangle$ and the symmetric group $S_n=\langle \rho_1, \rho_2,\ldots, \rho_{n-1} \rangle$  with the following relations:
\begin{align*}
\sigma_i \sigma_{i+1} \sigma_i&=\sigma_{i+1} \sigma_i \sigma_{i+1} & i=1, 2, \ldots, {n-2}, \\
\sigma_i \sigma_j&=\sigma_j \sigma_i &  |i-j| \geq 2, \\
\rho_i^{2}&=1 &  i=1, 2, \ldots, {n-1},\\
\rho_i \rho_j&= \rho_j \rho_i &  |i-j| \geq 2,\\
\rho_i \rho_{i+1} \rho_i&= \rho_{i+1} \rho_{i} \rho_{i+1} & i=1, 2, \ldots, {n-2},\\
\sigma_i \rho_j&= \rho_j \sigma_i & |i-j| \geq 2 ,\\
\rho_i \rho_{i+1} \sigma_i&= \sigma_{i+1} \rho_i \rho_{i+1} & i=1, 2, \ldots, {n-2}.
\end{align*}

The virtual pure braid group $VP_n$, $n\geq 2$, was introduced in  \cite{B} as the kernel of the homomorphism $VB_n \to S_n$, $\sigma_i \mapsto \rho_i$, 
$\rho_i \mapsto \rho_i$ for all $i=1, 2, \ldots, i-1$. $VP_n$
 admits a
presentation with the  generators $\lambda_{ij},\ 1\leq i\neq j\leq n,$
and the following relations:
\begin{align}
& \lambda_{ij}\lambda_{kl}=\lambda_{kl}\lambda_{ij} \label{rel},\\
&
\lambda_{ki}\lambda_{kj}\lambda_{ij}=\lambda_{ij}\lambda_{kj}\lambda_{ki}
\label{relation},
\end{align}
where distinct letters stand for distinct indices.
The generators of $VP_n$ can be expressed  in terms of  the generators of $VB_n$ by the formulas 
\begin{align*}
\lambda_{i,i+1} &= \rho_i \, \sigma_i,\\
  \lambda_{i+1,i} &= \rho_i \, \lambda_{i,i+1} \, \rho_i = \sigma_i \, \rho_i
\end{align*}
for  $i=1, 2, \ldots, n-1$, and
\begin{align*}
\lambda_{i,j} & = \rho_{j-1} \, \rho_{j-2} \ldots \rho_{i+1} \, \lambda_{i,i+1} \, \rho_{i+1} \ldots \rho_{j-2} \, \rho_{j-1}, \\ 
\lambda_{j,i} & =  \rho_{j-1} \, \rho_{j-2} \ldots \rho_{i+1} \, \lambda_{i+1,i} \, \rho_{i+1} \ldots \rho_{j-2} \, \rho_{j-1}
\end{align*} 
for $1 \leq i < j-1 \leq n-1$.


We have decomposition $VB_n = VP_n \rtimes S_n$ and $S_n$ acts on $VP_n$ by the rules

\begin{lemma}[\cite{B}] \label{form}
Let $a$ be an element of $\langle \rho_1, \rho_2, \ldots, \rho_{n-1} \rangle$ and $\bar{a}$ is its image in $S_n$ under the isomorphism $\rho_i \mapsto (i,i+1)$, $i = 1, 2, \ldots, n-1$, then for any generator $\lambda_{ij}$ of $VP_n$ the following holds
$$
a^{-1} \lambda_{ij} a = \lambda_{(i)\bar{a}, (j)\bar{a}},
$$
where $(k)\bar{a}$ is the image of $k$ under the action of the permutation $\bar{a}$.
\end{lemma}

\medskip

It is possible to define another epimorphism $\mu  \colon VB_n \to S_n$ as follows:
$$
\mu(\sigma_i)=e, \; \mu(\rho_i)=\rho_i, \;i=1,2,\dots, n-1\, ,
$$
where $S_n$ is generated by $\rho_i$ for $i=1,2,\dots, n-1$.
Let us denote by $VH_n$ the normal closure of $B_n$ in $VB_n$.
It is evident that $\ker \mu$ coincides with $VH_n$.
Let us define elements:
$$
x_{i,i+1}=\sigma_i,~~x_{i+1,i}=\rho_i \sigma_i \rho_i =\rho_i x_{i,i+1} \rho_i,
$$
for $i= 1, 2, \ldots, n-1$, and 
$$
x_{i,j}=\rho_{j-1} \cdots \rho_{i+1} \sigma_i \rho_{i+1} \cdots \rho_{j-1},
$$
$$
x_{j,i}=\rho_{j-1} \cdots \rho_{i+1} \rho_i \sigma_i \rho_i \rho_{i+1} \cdots \rho_{j-1},
$$
for $1 \le i < j-1 \le n-1$.

The group $VH_n$ admits a presentation with the generators $x_{k,\, l},$ $1 \leq k \neq l \leq
n$,
and the defining relations:
\begin{equation} \label{eq40}
x_{i,j} \,  x_{k,\, l} = x_{k,\, l}  \, x_{i,j},
\end{equation}
\begin{equation} \label{eq41}
x_{i,k} \,  x_{k,j} \,  x_{i,k} =  x_{k,j} \,  x_{i,k} \, x_{k,j},
\end{equation}
where  distinct letters stand for distinct indices.

This presentation was found in \cite{R} (see also \cite{BB}). We have decomposition
$VB_n = VH_n \rtimes S_n$, where  $S_n=\langle \rho_1, \dots, \rho_{n-1} \rangle$ acts on the generators of $VH_n$
 by permutation of indices and we have the next analogous of Lemma \ref{form}.

\begin{lemma}[\cite{BB}] \label{form1}
Let $a$ be an element of $\langle \rho_1, \rho_2, \ldots, \rho_{n-1} \rangle$ and $\bar{a}$ is its image in the symmetric group $S_n$ under the isomorphism $\rho_i \mapsto (i,i+1)$, $i = 1, 2, \ldots, n-1$, then for any generator $x_{ij}$ of $VH_n$ the following holds
$$
a^{-1} x_{ij} a = x_{(i)\bar{a}, (j)\bar{a}},
$$
where $(k)\bar{a}$ is the image of $k$ under the action of the permutation $\bar{a}$.
\end{lemma}
\par

\subsection{The twisted virtual braid group}\label{tvbg} This group $TVB_n$ was defined in  \cite{NPK}. We get this group if add to the presentation of $VB_n$ the set of new generators,
$$\gamma_1, \gamma_2, \ldots, \gamma_n,$$
which satisfy the relations
\begin{equation}\label{rel-inverse-b}
    \gamma_i ^2 = e~~ \text{for}~~i=1,\ldots, n,
\end{equation}
\begin{equation}\label{rel-height-bb}
\gamma_i \gamma_j = \gamma_j \gamma_i~~\textrm{for}~~ i, j \in \{1,2,\ldots,n\}.
\end{equation}
Also, we add four types of mixed relations,
\begin{equation}\label{rel-height-bv}
    \gamma_j \rho_i = \rho_i \gamma_j~~\textrm{for}~~  j \not= i, i+1,
\end{equation}
\begin{equation}\label{rel-height-sb}
   \gamma_j \sigma_i = \sigma_i \gamma_j~~\textrm{for}~~  j \not= i, i+1,
\end{equation}
\begin{equation}\label{rel-bv}
    \rho_i \gamma_i = \gamma_{i+1} \rho_i,~~\textrm{for}~~ i \in \{1,2,\ldots,n-1\},
\end{equation}
\begin{equation}\label{rel-twist-III}
   ~~\rho_i \sigma_i \rho_i  = \gamma_{i+1} \gamma_{i} \, \sigma_i \, \gamma_{i}  \, \gamma_{i+1} ~~\textrm{for}~~ i \in \{1,2,\ldots,n-1\}.
\end{equation}

From this presentation we see that  the subgroup $A_n = \langle \gamma_1, \gamma_2, \ldots, \gamma_n \rangle$ of $TVB_n$ is the quotient of the abelian group $\mathbb{Z}_2^n$. Later (see Corollary \ref{isom}), we show that in fact, $A_n \cong \mathbb{Z}_2^n$.
\begin{figure}[ht]
  \centering
    \includegraphics[width=12cm,height=7cm]{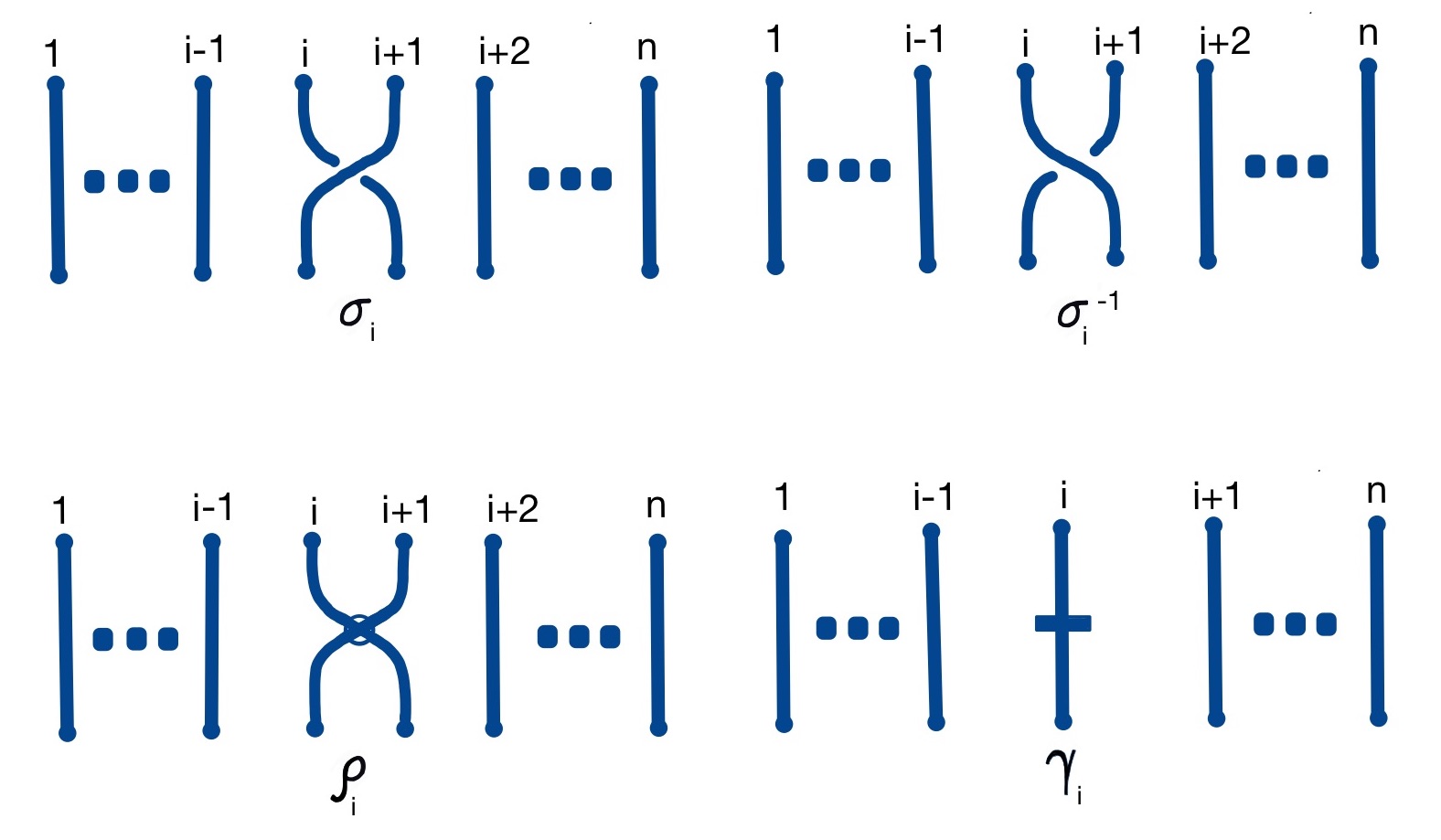}
        \caption{Generators of the group of twisted virtual braids}
        \label{gen}
        \end{figure}

\bigskip


\section{Some subgroups of $TVB_n$}  \label{pure}

One can define some epimorphism of $TVB_n$ onto $S_n$. The first one is
$$
\varphi_P \colon TVB_n \to S_n,~~\sigma_i \mapsto \rho_i, ~~\rho_i \mapsto \rho_i,~~i=1, 2, \ldots,n-1,~~\gamma_j \mapsto e,~~j=1, 2, \ldots, n.
$$
Its kernel $\ker(\varphi_P)$ is the twisted virtual pure braid group $TVP_n$.

\medskip
The second epimorphism  is
$$
\varphi_H \colon TVB_n \to S_n,~~\sigma_i \mapsto e, ~~\rho_i \mapsto \rho_i,~~i=1, 2, \ldots,n-1,~~\gamma_j \mapsto e,~~j=1, 2, \ldots, n.
$$
Its kernel $\ker(\varphi_H)$ is denoted by $TVH_n$.

It is clear that $TVP_n$  and $TVH_n$ are a normal subgroup of index $n!$ of $TVB_n$ and we have  short exact sequences,
$$
1 \to TVP_n \to TVB_n \to S_n \to 1,
$$
$$
1 \to  TVH_n \to  TVB_n \to S_n \to 1.
$$

\subsection{Structure of twisted virtual pure braid group $TVP_n$.}
 
Consider the following elements:
$$\lambda_{i,i+1}=\rho_i\sigma_i^{-1},\text{   } \lambda_{i+1,i}=\rho_i\lambda_{i,i+1}\rho_i, \text{   }  i=1, 2, \ldots, n-1,$$
$$\lambda_{i,j}=\rho_{j-1}\rho_{j-2}\cdots\rho_{i+1}\lambda_{i,i+1}\rho_{i+1}\cdots\rho_{j-2}\rho_{j-1},$$
$$\lambda_{j,i}=\rho_{j-1}\rho_{j-2}\cdots\rho_{i+1}\lambda_{i+1,i}\rho_{i+1}\cdots\rho_{j-2}\rho_{j-1},\text{   }  1 \leq i \leq j-1 \leq n-1,$$
which generate $VP_n$.

Before proving presentation of $TVP_n$, we require some results. The next lemma is analogous to Lemma \ref{form} and Lemma \ref{form1}.

\begin{lemma}\label{gamma}
The group $S_n$ acts by conjugation on the set $\{\gamma_i~ |~ 1\leq i\leq n\}$. This action is transitive. Let $a$ be an element of $\langle \rho_1, \rho_2, \ldots, \rho_{n-1} \rangle$ and $\bar{a}$ is its image in $S_n$ under the isomorphism $\rho_i \mapsto (i,i+1)$, $i = 1, 2, \ldots, n-1$, then for any generator $\gamma_{i}$ of $TVP_n$ the following holds
$$
a^{-1} \gamma_{i} a = \gamma_{(i)\bar{a}},
$$
where $(k)\bar{a}$ is the image of $k$ under the action of the permutation $\bar{a}$.
\end{lemma}

\begin{proof}
Using relations  (\ref{rel-inverse-b}), (\ref{rel-height-bv}), and (\ref{rel-bv}), it is easy to see that  the following 
 conjugation rules hold in $TVB_n$:

\begin{itemize}
    \item[(i)] $\rho_i\gamma_i\rho_i=\gamma_{i+1}  \Leftrightarrow \rho_i\gamma_{i+1}\rho_i=\gamma_{i}$, for $ 1\leq i < n$;
    \item[(ii)] $\rho_k\gamma_i\rho_k=\gamma_{i}$ for $ 1\leq i\neq k\leq n$.
\end{itemize}
The lemma follows from these rules. 
\end{proof}

\begin{remark}\label{remarkgamma}
To simplify the notation we will write  $a^{-1} \gamma_i  a =\gamma_{a(i)}$ instead $a^{-1} \gamma_{i} a = \gamma_{(i)\bar{a}}$ for $a \in S_n$.
\end{remark}

To find generators and  defining relations for  $TVP_n$, we use the Reidemeister-Schreier method (see, for example  \cite[Chapter 2.3]{MKS}). As a Schreier set of coset representation of $TVP_n$ in $TVB_n$ we take the same set $\Lambda_n$, which is used  in $VB_n$,
$$
\Lambda_n = \left\{ \prod\limits_{k=2}^n m_{k,j_k}~ |~ 1 \leq j_k
\leq k \right\}
$$
where $m_{kl}=\rho_{k-1}\rho_{k-2}\cdots \rho_l$ for $l<k$ and $m_{kl}=1$ in the other cases. 

The first main result of the present section is

\begin{theorem}\label{sTVP_n}
    The group $TVP_n$ admits a presentation with the generators \\ $\lambda_{kl},~1 \leq k\neq l \leq n$, and $\gamma_j$, $1\leq j\leq n$. The defining relations are as follows:
     \begin{align}
\lambda_{ij}\lambda_{kl} &=\lambda_{kl}\lambda_{ij},\label{comm-clas}\\
\lambda_{ki}(\lambda_{kj}\lambda_{ij}) &=(\lambda_{ij}\lambda_{kj})\lambda_{ki},\label{classical}\\
\gamma_i^2 &=1,\label{g1}\\
\gamma_i\gamma_j &=\gamma_j\gamma_i,\label{g2}\\
\lambda_{ij}\gamma_k & =\gamma_k\lambda_{ij},\label{g3}\\
\lambda_{ij} &=\gamma_i\gamma_j\lambda_{ji}\gamma_j\gamma_i,\label{g4}
\end{align}
where distinct letters stand for distinct indices.
\end{theorem}

\begin{proof}\label{main}
    Define the map $\Bar{}: TVB_n \to \Lambda_n$ which takes an element $w \in TVB_n$ to its representative $\overline{w}$ from $\Lambda_n$. In this case the element $w\overline{w}^{-1}$ belongs to $TVP_n$. By Theorem 2.7 of~\cite{MKS} the group $TVP_n$ is generated by
    $$
    s_{\lambda,a}=\lambda a \cdot (\overline{\lambda a})^{-1},~~~ \lambda \in \Lambda_n, ~~a \in \{\rho_1, \ldots, \rho_{n-1}, \sigma_1, \ldots, \sigma_{n-1}, \gamma_1, \ldots, \gamma_ n\}.
    $$
Then all   $s_{\lambda, \rho_i}=e$ and $s_{\lambda, \sigma_i}=\lambda (s_{e,\sigma_i} )\lambda^{-1}=\lambda (\sigma_i\rho_i)\lambda^{-1}=\lambda (\lambda_{i,i+1}^{-1})\lambda^{-1}$, which is equal to some $\lambda_{kl}$, by Lemma~\ref{form}.
   These calculations are done in Theorem 1~\cite{B}.
   
    Now, consider the generators $$s_{\lambda,\gamma_i}=\lambda(s_{e,\gamma_i})\lambda^{-1}.$$
    Since $s_{e,\gamma_i}=\gamma_i$. $s_{\lambda,\gamma_i}=\lambda(\gamma_i)\lambda^{-1}$, which is equal to some $\gamma_j$ by Lemma~\ref{gamma}.
   Therefore, generators of the group $TVP_n$ are $\lambda_{kl}$, $1 \leq k\neq l \leq n$ and, $\gamma_j$, $1\leq j\leq n$.

   To find the defining relations of $TVP_n$, we define a rewriting process $\tau$. It helps to rewrite a word $u$ to $\tau(u)$, where $u$ is written in the  generators of $TVB_n$ but represents an element of $TVP_n$ and $\tau(u)$ is a word written in  the generators of $TVP_n$. Let us associate to reduce word 
$$
u=a_1^{\epsilon_1}a_2^{\epsilon_2}\cdots a_v^{\epsilon_v}, ~~\epsilon_l=\pm 1, ~~a_l \in \{\sigma_1,\sigma_2, \ldots, \sigma_{n-1}, \rho_1,\rho_2, \ldots, \rho_{n-1}, \gamma_1,\gamma_2, \ldots, \gamma_n\},
$$
 the word
   $$\tau(u)=s_{k_1,a_1}^{\epsilon_1}s_{k_2,a_2}^{\epsilon_2}\cdots s_{k_v,a_v}^{\epsilon_v},$$
   in the generators of $TVP_n$, where $k_j$ is the $(j-1)$-th initial segment of the word $u$ if $\epsilon_j=1$, and a representative of the $j$-th initial segment of $u$ if $\epsilon_j=-1$.

   By Theorem 2.9 in~\cite{MKS}, the group $TVP_n$ is defined by the relations
   $$r_{\mu, \lambda}=\tau(\lambda r_\mu \lambda^{-1})=\lambda\tau( r_\mu) \lambda^{-1}, ~~~\lambda \in \Lambda_n,$$
   where $r_\mu$ is a defining relation of $TVB_n$.

   Common relations of $VB_n$ and $TVB_n$ give rise to relations (\ref{comm-clas}), and (\ref{classical}) proved in Theorem 1~\cite{B}.

Let us consider added relations stated in subsection~\ref{tvbg}.
   Denote by $r_1=\gamma_i^2$ the first   relation of $TVB_n$ which is not a relation of  $VB_n$. Then
   \begin{align*}
       r_{1,e} = \tau(r_1)& = s_{e,\gamma_i}s_{\bar{\gamma_i}, \gamma_i}\\
                         & = s_{e,\gamma_i}s_{e, \gamma_i}\\
                            & = (e\cdot\gamma_i(\overline{e\cdot\gamma_i})^{-1})^2\\
                            & = \gamma_i^2.
   \end{align*}
   The remaining $r_{1,\lambda}$, $\lambda \in \Lambda_n$, can be obtained from this relation using conjugation by $\lambda^{-1}$ and it gives the same relation, by Lemma~\ref{gamma}. We have obtained (\ref{g1}).

   Now, consider the next relation $r_2=\gamma_i\gamma_j\gamma_i\gamma_{j}.$
   We have,  \begin{align*}
              r_{2,e} = \tau(r_2)& = s_{e,\gamma_i}s_{\bar{\gamma_i},\gamma_j}s_{\overline{\gamma_i\gamma_j},\gamma_i}s_{\overline{\gamma_i\gamma_j\gamma_i}, \gamma_j}\\
                         & = s_{e,\gamma_i}s_{e,\gamma_j}s_{e, \gamma_i}s_{e,\gamma_j}\\
                        & = \gamma_i \gamma_j \gamma_i \gamma_j.
            \end{align*}
The remaining $r_{2,\lambda}$, $\lambda \in \Lambda_n$, will give the same relations. We have obtained (\ref{g2}).

Consider the next relation $r_3= \sigma_i\gamma_k\sigma_i^{-1}\gamma_k$, $k\neq i,i+1$,
\begin{align*}
           r_{3,e} = \tau(r_3)& = s_{e,\sigma_i}s_{\bar{\sigma_i},\gamma_k}s^{-1}_{\overline{\sigma_i\gamma_k\sigma_i^{-1}},\sigma_i}s_{\overline{ \sigma_i\gamma_k\sigma_i^{-1}}, \gamma_k} \\
                         & = s_{e,\sigma_i}s_{\rho_i,\gamma_k}s^{-1}_{e, \sigma_i}s_{e,\gamma_k}\\
                        & = \lambda^{-1}_{i,i+1} \rho_i \gamma_k \rho_i\lambda_{i,i+1}\gamma_k\\
                        & = \lambda^{-1}_{i,i+1}\gamma_k\lambda_{i,i+1}\gamma_k.
            \end{align*}
            The remaining relations $r_{3,\lambda}, \lambda \in \Lambda_n$, can be obtained from this relation using conjugation by $\lambda^{-1}$. By Lemma~\ref{gamma}, and Lemma~\ref{form}, we obtain (\ref{g3}).

Let us consider the last relation $r_4=\rho_i\sigma_i\rho_i\gamma_{i+1}\gamma_i\sigma^{-1}_i\gamma_i\gamma_{i+1}$,
\begin{align*}
              r_{4,e} = \tau(r_4)= & s_{e,\rho_i}s_{\bar{\rho_i},\sigma_i}s_{\overline{\rho_i\sigma_i},\rho_i},s_{\overline{\rho_i\sigma_i\rho_i},\gamma_{i+1}}s_{\overline{\rho_i\sigma_i\rho_i\gamma_{i+1}},\gamma_i}s^{-1}_{\overline{\rho_i\sigma_i\rho_i\gamma_{i+1}\gamma_i\sigma^{-1}_i},\sigma_i}\\
 & s_{\overline{\rho_i\sigma_i\rho_i\gamma_{i+1}\gamma_i\sigma^{-1}_i},\gamma_i}s_{\overline{\rho_i\sigma_i\rho_i\gamma_{i+1}\gamma_i\sigma^{-1}_i\gamma_i},\gamma_{i+1}}\\
                         = & 
  s_{e,\rho_i}s_{\rho_i,\sigma_i}s_{e,\rho_i}s_{\rho_i,\gamma_{i+1}}s_{\rho_i,\gamma_{i}}s^{-1}_{e, \sigma_i}s_{e,\gamma_i}s_{e,\gamma_{i+1}}\\
                        =  & \lambda^{-1}_{i+1,i} \gamma_{i}\gamma_{i+1}\lambda_{i,i+1}\gamma_i\gamma_{i+1}.
            \end{align*}
Conjugating this relation by all representatives from $\Lambda_n$, we obtain (\ref{g4}).

Therefore, the group $TVP_n$ is defined by the relations (\ref{comm-clas})-(\ref{g4}).
\end{proof}
We can establish an alternative representation of the group $TVP_n$ by applying the relations (\ref{g4}) as outlined in Theorem \ref{sTVP_n}.

\begin{corollary}\label{1sTVP_n}
    The group $TVP_n$ admits a presentation with the generators \\ $\lambda_{kl},~1 \leq k< l \leq n$, and $\gamma_j$, $1\leq j\leq n$. The defining relations are as follows:
     \begin{align}
\lambda_{ij}\lambda_{kl} &=\lambda_{kl}\lambda_{ij},\label{1comm-clas}\\
\lambda_{ki}(\lambda_{kj}\lambda_{ij}) &=(\lambda_{ij}\lambda_{kj})\lambda_{ki},\label{classical1}\\
\lambda_{ki}^{(ki)}(\lambda_{kj}^{(kj)}\lambda_{ij}^{(ij)}) &=(\lambda_{ij}^{(ij)}\lambda_{kj}^{(kj)})\lambda_{ki}^{(ki)},\label{classical2}\\
\lambda_{ki}^{(ki)}(\lambda_{ij}\lambda_{kj}) &=(\lambda_{kj}\lambda_{ij})\lambda_{ki}^{(ki)},\label{classical3}\\
\lambda_{ki}(\lambda_{ij}^{(ij)}\lambda_{kj}^{(kj)}) &=(\lambda_{kj}^{(kj)}\lambda_{ij}^{(ij)})\lambda_{ki},\label{classical4}\\
\lambda_{kj}(\lambda_{ki}\lambda_{ij}^{(ij)}) &=(\lambda_{ij}^{(ij)}\lambda_{ki})\lambda_{kj},\label{classical5}\\
\lambda_{kj}^{(kj)}(\lambda_{ki}^{(ki)}\lambda_{ij}) &=(\lambda_{ij}\lambda_{ki}^{(ki)})\lambda_{kj}^{(kj)},\label{classical6}\\
\gamma_i^2 &=1,\label{1g1}\\
\gamma_i\gamma_j &=\gamma_j\gamma_i,\label{1g2}\\
\lambda_{ij}\gamma_k & =\gamma_k\lambda_{ij}\label{1g3},
\end{align}
 where we denoted $\lambda_{ij}^{(ij)}=\lambda_{ij}^{\gamma_i\gamma_j}$ and, as usual,  distinct letters stand for distinct indices.
\end{corollary}

\begin{example}
Structure of $TVP_1$ and $TVP_2$ are shown as below:
\begin{itemize}
    \item[1)~~~] $TVP_1=\langle \gamma_1~|~ \gamma_1^2 = e \rangle \cong \mathbb{Z}_2$.
    
    \item[2)~~~] $TVP_2= \langle \lambda_{12},  \lambda_{21}, \gamma_1, \gamma_2~ |~ 
\lambda_{12}=\gamma_1 \gamma_2 \lambda_{21} \gamma_2 \gamma_1, ~~\gamma_1^2=\gamma_2^2=e, ~~\gamma_1\gamma_2=\gamma_2\gamma_1 \rangle.$
We can remove the generator $\lambda_{21}$ and the first relation. Then,
$$
TVP_2= \langle \lambda_{12}, \gamma_1, \gamma_2~ |~ \gamma_1^2=\gamma_2^2=e, ~~\gamma_1\gamma_2=\gamma_2\gamma_1, \rangle \cong \mathbb{Z} * (\mathbb{Z}_2 \times \mathbb{Z}_2).
$$
\end{itemize}
\end{example}

From the definition of $TVP_n$, Lemma~\ref{gamma}, and Lemma~\ref{form} it follows 

\begin{corollary}
$TVB_n=TVP_n \rtimes S_n.$
\end{corollary}

\color{black}
\subsection{Decomposition of $TVP_n$} There exists  an epimorphism 
$$
\psi_P \colon TVP_n \to A_n,~~\lambda_{kl} \mapsto e,~~\gamma_j \mapsto \gamma_j,
$$ 
where  
$$
A_n = \langle \gamma_1, \ldots, \gamma_n \rangle.
$$
Analysing the presentation of $TVP_n$ and the image of $\psi_P$, we get

\begin{corollary} \label{isom}
The subgroup $A_n$ of $TVP_n$ has a presentation
$$
A_n = \langle \gamma_1, \ldots, \gamma_n ~|~\gamma_1^2 = \ldots =  \gamma_n^2 = e, ~~\gamma_i \gamma_j = \gamma_j \gamma_i.~1 \leq i < j \leq n \rangle,
$$
i.e.  is isomorphic to  $\mathbb{Z}_2^n$.
\end{corollary}

Denote by $PL_n = \ker(\psi_P)$. Then $TVP_n = PL_n   \rtimes A_n$.

It is interesting to find a structure of $PL_n$. Using the relations (\ref{g4}) in Theorem \ref{sTVP_n},
$$
\lambda_{ji}=\gamma_i\gamma_j\lambda_{ij}\gamma_j\gamma_i,~~~1  \leq i < j \leq n,
$$
we can remove all the generators $\lambda_{ji}$, $1  \leq i < j \leq n,$ from the generating set of $TVP_n$.
In particular, in the case $n = 3$,  using the relations
$$
\lambda_{21} = \lambda_{12}^{ \gamma_1 \gamma_2},~~~\lambda_{31} = \lambda_{13}^{ \gamma_1 \gamma_3},~~~ 
\lambda_{32} = \lambda_{23}^{ \gamma_2 \gamma_3},
$$ 
 we get
$$
TVP_3 =  \langle \lambda_{12}, \lambda_{13}, \lambda_{23},  \gamma_1, \gamma_2, \gamma_3~|~
$$
$$\lambda_{12}\lambda_{13}\lambda_{23}=\lambda_{23}\lambda_{13}\lambda_{12}\text{, \hspace{0.3cm}}\lambda_{12}^{\gamma_1 \gamma_2} \lambda_{23}\lambda_{13}=\lambda_{13}\lambda_{23}\lambda_{12}^{\gamma_1 \gamma_2}\text{, \hspace{0.3cm}}\lambda_{13}\lambda_{12}\lambda_{23}^{\gamma_2 \gamma_3}=\lambda_{23}^{\gamma_2 \gamma_3} \lambda_{12}\lambda_{13},$$
$$\lambda_{13}^{\gamma_1 \gamma_3} \lambda_{23}^{\gamma_2 \gamma_3} \lambda_{12}=\lambda_{12}  \lambda_{23}^{\gamma_2 \gamma_3} 
\lambda_{13}^{\gamma_1 \gamma_3} \text{, \hspace{0.3cm}}\lambda_{23} \lambda_{12}^{\gamma_1 \gamma_2} \lambda_{13}^{\gamma_1 \gamma_3}=
\lambda_{13}^{\gamma_1 \gamma_3} \lambda_{12}^{\gamma_1 \gamma_2}\lambda_{23} \text{, \hspace{0.3cm}} \lambda_{23}^{\gamma_2 \gamma_3} \lambda_{13}^{\gamma_1 \gamma_3} \lambda_{12}^{\gamma_1 \gamma_2}=\lambda_{12}^{\gamma_1 \gamma_2}\lambda_{13}^{\gamma_1 \gamma_3} \lambda_{23}^{\gamma_2 \gamma_3},$$
$$ \text{\hspace{0.3cm}} \gamma_1\lambda_{23}=\lambda_{23} \gamma_1 \text{, \hspace{0.3cm}} \gamma_1 \lambda_{23}^{\gamma_2 \gamma_3}=\lambda_{23}^{\gamma_2 \gamma_3}\gamma_1 \text{, \hspace{0.3cm}} \gamma_2\lambda_{13}^{\gamma_1 \gamma_3} = \lambda_{13}^{\gamma_1 \gamma_3} \gamma_2,  $$
$$\text{\hspace{0.3cm}} \gamma_2\lambda_{13}=\lambda_{13} \gamma_2 \text{, \hspace{0.3cm}} \gamma_3 \lambda_{12} = \lambda_{12} \gamma_3 \text{, \hspace{0.3cm}} \gamma_3\lambda_{12}^{\gamma_1 \gamma_2}=\lambda_{12}^{\gamma_1 \gamma_2}\gamma_3,~~rel(A_3) 
\rangle,
$$
where $rel(A_3)$ means the set of defining relations in $A_3$. The relations
$$  \gamma_1 \lambda_{23}^{\gamma_2 \gamma_3}=\lambda_{23}^{\gamma_2 \gamma_3}\gamma_1 \text{, \hspace{0.3cm}} \gamma_2\lambda_{13}^{\gamma_1 \gamma_3} = \lambda_{13}^{\gamma_1 \gamma_3} \gamma_2, 
\text{\hspace{0.3cm}}  \gamma_3\lambda_{12}^{\gamma_1 \gamma_2}=\lambda_{12}^{\gamma_1 \gamma_2}\gamma_3
$$
follows from other relations and we can remove them. We get
$$
TVP_3 =  \langle \lambda_{12}, \lambda_{13}, \lambda_{23},  \gamma_1, \gamma_2, \gamma_3~|~
$$
$$\lambda_{12}\lambda_{13}\lambda_{23}=\lambda_{23}\lambda_{13}\lambda_{12}\text{, \hspace{0.3cm}}
\lambda_{12}^{\gamma_1 \gamma_2}\lambda_{13}^{\gamma_1 \gamma_3} \lambda_{23}^{\gamma_2 \gamma_3}=\lambda_{23}^{\gamma_2 \gamma_3} \lambda_{13}^{\gamma_1 \gamma_3} \lambda_{12}^{\gamma_1 \gamma_2},
$$

$$
\lambda_{12}^{\gamma_1 \gamma_2} \lambda_{23}\lambda_{13}=\lambda_{13}\lambda_{23}\lambda_{12}^{\gamma_1 \gamma_2}\text{, \hspace{0.3cm}}
 \lambda_{12}  \lambda_{23}^{\gamma_2 \gamma_3} 
\lambda_{13}^{\gamma_1 \gamma_3}  =\lambda_{13}^{\gamma_1 \gamma_3} \lambda_{23}^{\gamma_2 \gamma_3} \lambda_{12},
$$

$$
\lambda_{13}\lambda_{12}\lambda_{23}^{\gamma_2 \gamma_3}=\lambda_{23}^{\gamma_2 \gamma_3} \lambda_{12}\lambda_{13}\text{, \hspace{0.3cm}}
\lambda_{13}^{\gamma_1 \gamma_3} \lambda_{12}^{\gamma_1 \gamma_2}\lambda_{23}=\lambda_{23} \lambda_{12}^{\gamma_1 \gamma_2} \lambda_{13}^{\gamma_1 \gamma_3},
$$

$$\gamma_1\lambda_{23}=\lambda_{23} \gamma_1 \text{, \hspace{0.3cm}}  \gamma_2\lambda_{13}=\lambda_{13} \gamma_2 \text{, \hspace{0.3cm}} \gamma_3 \lambda_{12} = \lambda_{12} \gamma_3 \text{, \hspace{0.3cm}}~~rel(A_3) 
\rangle.
$$

Let us introduce elements
$$
\lambda_{12}^{(0)} =  \lambda_{12},~~ \lambda_{12}^{(1)} =  \lambda_{12}^{\gamma_1},~~\lambda_{12}^{(2)} =  \lambda_{12}^{\gamma_2},
~~\lambda_{12}^{(12)} =  \lambda_{12}^{\gamma_1\gamma_2},
$$
$$
\lambda_{13}^{(0)} =  \lambda_{13},~~\lambda_{13}^{(1)} =  \lambda_{13}^{\gamma_1},~~\lambda_{13}^{(3)} =  \lambda_{13}^{\gamma_3},
~~\lambda_{13}^{(13)} =  \lambda_{13}^{\gamma_1\gamma_3},
$$
$$
\lambda_{23}^{(0)} =  \lambda_{23},~~\lambda_{23}^{(2)} =  \lambda_{23}^{\gamma_2},~~\lambda_{23}^{(3)} =  \lambda_{23}^{\gamma_3},
~~\lambda_{23}^{(23)} =  \lambda_{23}^{\gamma_2\gamma_3}.
$$

Using the Reidemeister-Schreier method it is not difficult to prove

\begin{proposition}
The group $PL_3$ is generated by elements 
$$
\lambda_{ij}^{(0)}, ~\lambda_{ij}^{(i)},~ \lambda_{ij}^{(j)},~ \lambda_{ij}^{(ij)}, ~~1\leq i<j \leq 3,
$$
 and is defined by the relations which can be found from the relations
$$
\lambda_{12}^{(0)} \lambda_{13}^{(0)} \lambda_{23}^{(0)} =\lambda_{23} ^{(0)} \lambda_{13} ^{(0)} \lambda_{12}^{(0)} \text{, \hspace{0.3cm}}
\lambda_{12}^{(12)}\lambda_{13}^{(13)} \lambda_{23}^{(23)}=\lambda_{23}^{(23)} \lambda_{13}^{(13)}\lambda_{12}^{(12)},
$$

$$
\lambda_{12}^{(12)} \lambda_{23}^{(0)} \lambda_{13}^{(0)} =\lambda_{13}^{(0)} \lambda_{23}^{(0)} \lambda_{12}^{(12)}\text{, \hspace{0.3cm}}
 \lambda_{12}^{(0)}  \lambda_{23}^{(23)}
\lambda_{13}^{(13)} =\lambda_{13}^{(13)} \lambda_{23}^{(23)} \lambda_{12}^{(0)},
$$

$$
\lambda_{13}^{(0)} \lambda_{12}^{(0)} \lambda_{23}^{(23)}=\lambda_{23}^{(23)} \lambda_{12}^{(0)} \lambda_{13}^{(0)} \text{, \hspace{0.3cm}}
\lambda_{13}^{(13)} \lambda_{12}^{(12)}\lambda_{23}^{(0)} = \lambda_{23}^{(0)}  \lambda_{12}^{(12)} \lambda_{13}^{(13)},
$$
using the conjugations by elements
$$
\gamma_1^{\varepsilon_1} \gamma_2^{\varepsilon_2} \gamma_3^{\varepsilon_3},~~~\varepsilon_1, \varepsilon_2, \varepsilon_3 \in \{ 0, 1 \}.
$$
\end{proposition}

The general case can be established using the Reidemeister-Schreier method 
\begin{theorem} \label{PL_n}
The group $PL_n$, $n \geq 2$ is generated by elements

$$\lambda_{ij}^{(0)} = \lambda_{ij}, ~\lambda_{ij}^{(i)} = \lambda_{ij}^{\gamma_i},~ \lambda_{ij}^{(j)} = \lambda_{ij}^{\gamma_j},~ \lambda_{ij}^{(ij)} = \lambda_{ij}^{\gamma_i \gamma_j}, ~~1\leq i<j \leq n,$$

and relations are defined as follows,

$$\lambda_{ij}^{(0)}\lambda_{kl}^{(0)} =\lambda^{(0)}_{kl}\lambda^{(0)}_{ij},~~ \{i,j\}\cap \{k,l\}=\phi$$
$$\lambda^{(i)}_{ij}\lambda^{(0)}_{kl} =\lambda^{(0)}_{kl}\lambda^{(i)}_{ij}, ~~\lambda^{(j)}_{ij}\lambda^{(0)}_{kl} =\lambda^{(0)}_{kl}\lambda^{(j)}_{ij},~~ \lambda^{(0)}_{ij}\lambda^{(k)}_{kl} =\lambda^{(k)}_{kl}\lambda^{(0)}_{ij}$$
$$\lambda^{(0)}_{ij}\lambda^{(l)}_{kl} =\lambda_{kl}^{(l)}\lambda^{(0)}_{ij}, ~~\lambda^{(ij)}_{ij}\lambda^{(0)}_{kl} =\lambda^{(0)}_{kl}\lambda_{ij}^{(ij)},~~ \lambda^{(0)}_{ij}\lambda^{(kl)}_{kl} =\lambda_{kl}^{(kl)}\lambda^{(0)}_{ij}$$
$$\lambda^{(i)}_{ij}\lambda^{(k)}_{kl} =\lambda^{(k)}_{kl}\lambda^{(i)}_{ij}, ~~\lambda^{(i)}_{ij}\lambda^{(l)}_{kl} =\lambda^{(l)}_{kl}\lambda^{(i)}_{ij},~~ \lambda^{(j)}_{ij}\lambda^{(k)}_{kl} =\lambda^{(k)}_{kl}\lambda^{(j)}_{ij}$$
$$\lambda^{(j)}_{ij}\lambda^{(l)}_{kl} =\lambda^{(l)}_{kl}\lambda^{(j)}_{ij}, ~~\lambda^{(ij)}_{ij}\lambda^{(k)}_{kl} =\lambda^{(k)}_{kl}\lambda^{(ij)}_{ij},~~ \lambda^{(ij)}_{ij}\lambda^{(l)}_{kl} =\lambda^{(l)}_{kl}\lambda^{(ij)}_{ij}$$
$$\lambda^{(i)}_{ij}\lambda^{(kl)}_{kl} =\lambda^{(kl)}_{kl}\lambda^{(i)}_{ij}, ~~\lambda^{(j)}_{ij}\lambda^{(kl)}_{kl} =\lambda^{(kl)}_{kl}\lambda^{(j)}_{ij},~~ \lambda^{(ij)}_{ij}\lambda^{(kl)}_{kl} =\lambda^{(kl)}_{kl}\lambda^{(ij)}_{ij}$$

$$\lambda_{ij}^{(0)} \lambda_{ik}^{(0)} \lambda_{jk}^{(0)} =\lambda_{jk}^{(0)} \lambda_{ik}^{(0)} \lambda_{ij}^{(0)} \text{, \hspace{0.3cm}}
\lambda_{ij}^{(i)} \lambda_{ik}^{(i)} \lambda_{jk}^{(0)} =\lambda_{jk} ^{(0)} 
\lambda_{ik} ^{(i)} \lambda_{ij}^{(i)} $$
$$\lambda_{ij}^{(j)} \lambda_{ik}^{(0)} \lambda_{jk}^{(j)} =\lambda_{jk}^{(j)} \lambda_{ik}^{(0)} \lambda_{ij}^{(j)} \text{, \hspace{0.3cm}}
\lambda_{ij}^{(0)} \lambda_{ik}^{(k)} \lambda_{jk}^{(k)} =\lambda_{jk} ^{(k)} \lambda_{ik} ^{(k)} \lambda_{ij}^{(0)} $$
$$\lambda_{ij}^{(ij)} \lambda_{ik}^{(i)} \lambda_{jk}^{(j)} =\lambda_{jk} ^{(j)} \lambda_{ik} ^{(i)} \lambda_{ij}^{(ij)} \text{, \hspace{0.3cm}}
\lambda_{ij}^{(j)} \lambda_{ik}^{(k)} \lambda_{jk}^{(jk)} =\lambda_{jk} ^{(jk)} \lambda_{ik} ^{(k)} \lambda_{ij}^{(j)} $$
$$\lambda_{ij}^{(i)} \lambda_{ik}^{(ik)} \lambda_{jk}^{(k)} =\lambda_{jk} ^{(k)} \lambda_{ik} ^{(ik)} \lambda_{ij}^{(i)} \text{, \hspace{0.3cm}}
\lambda_{ij}^{(ij)}\lambda_{ik}^{(ik)} \lambda_{jk}^{(jk)}=\lambda_{jk}^{(jk)} \lambda_{ik}^{(ik)}\lambda_{ij}^{(ij)}$$

$$\lambda_{ij}^{(ij)} \lambda_{jk}^{(0)} \lambda_{ik}^{(0)} =\lambda_{ik}^{(0)} \lambda_{jk}^{(0)} \lambda_{ij}^{(ij)}\text{, \hspace{0.3cm}}
\lambda_{ij}^{(j)} \lambda_{jk}^{(0)} \lambda_{ik}^{(i)} =\lambda_{ik}^{(i)} \lambda_{jk}^{(0)} \lambda_{ij}^{(j)}$$
$$ \lambda_{ij}^{(i)}  \lambda_{jk}^{(j)}\lambda_{ik}^{(0)} =\lambda_{ik}^{(0)} \lambda_{jk}^{(j)} \lambda_{ij}^{(i)}\text{, \hspace{0.3cm}}  
 \lambda_{ij}^{(0)}  \lambda_{jk}^{(j)}\lambda_{ik}^{(i)} =\lambda_{ik}^{(i)} \lambda_{jk}^{(j)} \lambda_{ij}^{(0)}$$
 $$\lambda_{ij}^{(ij)} \lambda_{jk}^{(k)} \lambda_{ik}^{(k)} =\lambda_{ik}^{(k)} \lambda_{jk}^{(k)} \lambda_{ij}^{(ij)}\text{, \hspace{0.3cm}}
\lambda_{ij}^{(i)} \lambda_{jk}^{(jk)} \lambda_{ik}^{(k)} =\lambda_{ik}^{(k)} \lambda_{jk}^{(jk)} \lambda_{ij}^{(i)}$$
$$ \lambda_{ij}^{(j)}  \lambda_{jk}^{(k)}\lambda_{ik}^{(ik)} =\lambda_{ik}^{(ik)} \lambda_{jk}^{(k)} \lambda_{ij}^{(j)}\text{, \hspace{0.3cm}}  
 \lambda_{ij}^{(0)}  \lambda_{jk}^{(jk)}\lambda_{ik}^{(ik)} =\lambda_{ik}^{(ik)} \lambda_{jk}^{(jk)} \lambda_{ij}^{(0)}$$
 
$$\lambda_{ik}^{(0)} \lambda_{ij}^{(0)} \lambda_{jk}^{(jk)}=\lambda_{jk}^{(jk)} \lambda_{ij}^{(0)} \lambda_{ik}^{(0)} \text{, \hspace{0.3cm}}
\lambda_{ik}^{(i)} \lambda_{ij}^{(i)} \lambda_{jk}^{(jk)}=\lambda_{jk}^{(jk)} \lambda_{ij}^{(i)} \lambda_{ik}^{(i)} $$
$$\lambda_{ik}^{(0)} \lambda_{ij}^{(j)} \lambda_{jk}^{(k)}=\lambda_{jk}^{(k)} \lambda_{ij}^{(j)} \lambda_{ik}^{(0)} \text{, \hspace{0.3cm}}
\lambda_{ik}^{(k)} \lambda_{ij}^{(0)} \lambda_{jk}^{(j)}=\lambda_{jk}^{(j)} \lambda_{ij}^{(0)} \lambda_{ik}^{(k)} $$
$$\lambda_{ik}^{(i)} \lambda_{ij}^{(ij)} \lambda_{jk}^{(k)}=\lambda_{jk}^{(k)} \lambda_{ij}^{(ij)} \lambda_{ik}^{(i)} \text{, \hspace{0.3cm}}
\lambda_{ik}^{(k)} \lambda_{ij}^{(j)} \lambda_{jk}^{(jk)}=\lambda_{jk}^{(jk)} \lambda_{ij}^{(j)} \lambda_{ik}^{(j)} $$
$$\lambda_{ik}^{(ik)} \lambda_{ij}^{(i)}\lambda_{jk}^{(j)} = \lambda_{jk}^{(j)}  \lambda_{ij}^{(i)} \lambda_{ik}^{(ik)}\text{, \hspace{0.3cm}}
\lambda_{ik}^{(ik)} \lambda_{ij}^{(ij)}\lambda_{jk}^{(0)} = \lambda_{jk}^{(0)}  \lambda_{ij}^{(ij)} \lambda_{ik}^{(ik)}$$
\end{theorem}

\begin{theorem} The map 
$$
\lambda^{(0)}_{ij} \to \lambda_{ij},~~\lambda^{(j)}_{ij} \to e,~~\lambda^{(i)}_{ij} \to e,~~ \lambda^{(ij)}_{ij} \to \lambda_{ji},~~~1 \leq i < j \leq n,
$$
defines an endomorphism of $PL_n$ onto $VP_n$. Consequently, it follows $VP_n \leq PL_n$.
\end{theorem}
\begin{proof}
If we observe, the defined endomorphism's image contains the following non-trivial relations:
$$\lambda_{ij}\lambda_{kl} =\lambda_{kl}\lambda_{ij},~~\lambda_{ji}\lambda_{kl} =\lambda_{kl}\lambda_{ji}$$
$$\lambda_{ij}\lambda_{lk} =\lambda_{lk}\lambda_{ij},~~\lambda_{ji}\lambda_{lk} =\lambda_{lk}\lambda_{ji}$$
$$\lambda_{ij}\lambda_{ik} \lambda_{jk} =\lambda_{jk}\lambda_{ik} \lambda_{ij} \text{, \hspace{0.3cm}}
\lambda_{ji}\lambda_{ki} \lambda_{kj}=\lambda_{kj} \lambda_{ki}\lambda_{ji},$$
$$\lambda_{ji} \lambda_{jk} \lambda_{ik}=\lambda_{ik} \lambda_{jk} \lambda_{ji}\text{, \hspace{0.3cm}}
\lambda_{ij}  \lambda_{kj}\lambda_{ki} =\lambda_{ki} \lambda_{kj}\lambda_{ij}$$
$$\lambda_{ik} \lambda_{ij} \lambda_{kj}=\lambda_{kj}\lambda_{ij}\lambda_{ik}\text{, \hspace{0.3cm}}
\lambda_{ki}\lambda_{ji}\lambda_{jk} = \lambda_{jk} \lambda_{ji}\lambda_{ki}$$
All these relations satisfied in $VP_n$.
Hence $VP_n \leq PL_n$.
\end{proof}
As corollary we get

\begin{corollary}
The virtual pure braid group $VP_n$ is a subgroup of $TVP_n$.
\end{corollary}

\subsection{Presentation of  $TVH_n$}\label{tvhn}

Consider the following elements of  $TVB_n$:
$$x_{i,i+1}=\sigma_i,\text{   } x_{i+1,i}=\rho_i\sigma_i\rho_i, \text{   }  i=1,\ldots, n-1,$$
$$x_{i,j}=\rho_{j-1}\rho_{j-2}\cdots\rho_{i+1}\sigma_{i,i+1}\rho_{i+1}\cdots\rho_{j-2}\rho_{j-1},$$ 
$$ x_{j,i}=\rho_{j-1}\rho_{j-2}\cdots\rho_{i+1}\rho_i\sigma_i\rho_i\rho_{i+1}\cdots\rho_{j-2}\rho_{j-1},\text{   }  1 \leq i \leq j-1 \leq n-1.$$
\begin{theorem}\label{main2}
    The group $TVH_n$ admits a presentation with the generators \\ $x_{kl}$, $1 \leq k\neq l \leq n$, and $\gamma_j$, $1\leq j\leq n$. The defining relations are as follows:
     \begin{align}
x_{ij}x_{kl} &=x_{kl}x_{ij}\label{comm-clas2}\\
x_{ik}x_{kj}x_{ik} &=x_{kj}x_{ik}x_{kj}\label{2classical2}\\
\gamma_i^2 &=1\label{g5}\\
\gamma_i\gamma_j &=\gamma_j\gamma_i\label{g6}\\
x_{ij}\gamma_k & =\gamma_kx_{ij}\label{g7}\\
x_{ij} &=\gamma_i\gamma_jx_{ji}\gamma_j\gamma_i\label{g8},
\end{align}
where distinct letters stand for distinct indices.
\end{theorem}

\begin{proof}
It is similar to the proof of Theorem~\ref{sTVP_n}.
\end{proof}

We can establish an alternative representation of the group $TVH_n$ by applying the relations (\ref{g8}) as outlined in Theorem \ref{main2}. 
\begin{corollary}\label{1TVH_n}
    The group $TVH_n$ admits a presentation with the generators \\ $x_{kl},~1 \leq k< l \leq n$, and $\gamma_j$, $1\leq j\leq n$. The defining relations are as follows:
     \begin{align}
x_{ij} \, x_{kl} &=x_{kl} \, x_{ij},\label{2comm-clas}\\
x_{ik} \, x_{kj} \, x_{ik} &=x_{kj}\, x_{ik} \, x_{kj},\label{1classical}\\
x_{ij} \, x_{kj}^{(kj)} \, x_{ij} &=x_{kj}^{(kj)} \, x_{ij} \, x_{kj}^{(kj)},\label{2classical}\\
x_{ij}^{(ij)} \, x_{ik} \, x_{ij}^{(ij)} &=x_{ik} \, x_{ij}^{(ij)} \, x_{ik},\label{3classical}\\
x_{ik}^{(ik)} \, x_{kj}^{(kj)} \, x_{ik}^{(ik)} &=x_{kj}^{(kj)} \, x_{ik}^{(ik)} \, x_{kj}^{(kj)},\label{4classical}\\
\gamma_i^2 &=1,\label{2g1}\\
\gamma_i \, \gamma_j &=\gamma_j \, \gamma_i,\label{2g2}\\
x_{ij} \, \gamma_k & =\gamma_k \, \lambda_{ij}\label{2g3},
\end{align}
where we denoted $x^{(ij)}_{ij}=x^{\gamma_i\gamma_j}$ and, as usual, distinct letters stand for distinct indices.
\end{corollary}

It is easy to see that the map 
$$
\psi_H \colon TVH_n \to A_n = \langle \gamma_1, \gamma_2, \ldots, \gamma_n \rangle,
$$
which is defined on the generators $x_{kl} \mapsto e,~1 \leq k\neq l \leq n$, $\gamma_j \mapsto \gamma_j $, $1\leq j\leq n$, can be extended to an endomorphism of $TVH_n$.

Similar to the methodology used previously for $\ker(\psi_P)$ it is ease to prove

\begin{theorem} \label{HL_n}
The group  $HL_n = \ker(\psi_H)$ is generated by elements
$$
x_{ij}^{(0)} = x_{ij}, ~x_{ij}^{(i)} = x_{ij}^{\gamma_i},~ x_{ij}^{(j)} = x_{ij}^{\gamma_j},~ x_{ij}^{(ij)} = x_{ij}^{\gamma_i \gamma_j}, ~~1\leq i<j \leq n,
$$
and relations are generated by conjugating the relations (\ref{2comm-clas}) to (\ref{4classical}) with elements from $A_n$.
\end{theorem}

As corollary we get

\begin{corollary} 
 $TVH_n = HL_n \rtimes A_n$.
\end{corollary}

The natural question that arises is whether $TVP_n$ is isomorphic to $TVH_n$. We will establish this in the subsequent part.

It is not difficult to find,
$$
TVH_2= \langle x_{12},x_{21},\gamma_1,\gamma_2 | x_{12}=\gamma_1\gamma_2x_{21}\gamma_2\gamma_1, \gamma_1\gamma_2=\gamma_2\gamma_1, \gamma_1^2=\gamma_2^2=1\rangle.
$$
 Hence, $TVH_2 \cong TVP_2$.

But in general case  we have
\begin{proposition}
The group $TVH_n$ and $TVP_n$ are not isomorphic for $n\geq 3$.
\end{proposition}
\begin{proof}
It is sufficient to prove that the abelianisation of $TVP_n$ and $TVH_n$ are different. 

From the relation~(\ref{g4}) follows that in the  abelianisation of $TVP_n$, $\lambda_{ij}=\lambda_{ji}$. So, abelianisation of $TVP_n$ is isomorphic to $\mathbb{Z}^{\frac{n(n-1)}{2}}\bigoplus \mathbb{Z}_2^n$. 
By \cite[Proposition 19]{BB} and relation~(\ref{g8}) in the  abelianisation of $TVH_n$,  all $x_{ij}$ are identified. Hence, abelianisation of $TVH_n$ is $\mathbb{Z} \bigoplus \mathbb{Z}_2^n$.
\end{proof}

\medskip

\bigskip


\section{Maps $\varphi_{PT}$ and $\varphi_{HT} $} \label{PT}

Let us defined two endomorphism of $TVB_n$ by the actions on the generators,
$$
\varphi_{PT} \colon TVB_n \to TS_n,~~\sigma_i \mapsto \rho_i, ~~\rho_i \mapsto \rho_i,~~i=1, 2, \ldots,n-1,~~\gamma_j \mapsto \gamma_j,~~j=1, 2, \ldots, n,
$$
$$
\varphi_{HT} \colon TVB_n \to TS_n,~~\sigma_i \mapsto e, ~~\rho_i \mapsto \rho_i,~~i=1, 2, \ldots,n-1,~~\gamma_j \mapsto \gamma_j,~~j=1, 2, \ldots, n,
$$
where $TS_n$ is the subgroup of $TVB_n$ that is generated by
$$
\rho_1, \rho_2, \ldots, \rho_{n-1}, \gamma_1, \gamma_2, \ldots, \gamma_n.
$$

Considering the defining relations of $TVB_n$, we get

\begin{proposition}
$$
Im(\varphi_{PT}) = Im(\varphi_{HT}) = TS_n = \langle \rho_1, \rho_2, \ldots, \rho_{n-1}, \gamma_1, \gamma_2, \ldots, \gamma_n \rangle
$$
is isomorphic to the extended symmetric group and hence, has a presentation $TS_n = A_n \rtimes S_n$.
\end{proposition}

\begin{proof}
It is need to show that under these endomorphism any relation of $TVB_n$ goes to a relation which holds in extended symmetric group. For $\varphi_{PT}$ it is evident for all relations, except the relation (\ref{rel-twist-III}),
$$
\rho_i \sigma_i \rho_i  = \gamma_{i+1} \gamma_{i} \, \sigma_i \, \gamma_{i}  \, \gamma_{i+1}, 
$$   
which goes to the relation
$$
\rho_i  = \gamma_{i+1} \gamma_{i} \, \rho_i\, \gamma_{i}  \, \gamma_{i+1} ~\Leftrightarrow~e = (\rho_i  \gamma_{i+1} \gamma_{i} \, \rho_i ) \, \gamma_{i}  \, \gamma_{i+1}, 
$$
and using  (\ref{rel-bv}),
$$
    \rho_i \gamma_i = \gamma_{i+1} \rho_i,
$$
we get 
$$
e = \gamma_{i} \gamma_{i+1} \, \gamma_{i}  \, \gamma_{i+1}. 
$$
Since $\gamma_{i}$ and $\gamma_{i+1}$ are commute, we get the trivial relation.

For the endomorphism $\varphi_{HT}$ the claim is evident.
\end{proof}

Further  we will find sets of generators and defining relations for $\ker(\varphi_{PT})$ and $\ker(\varphi_{HT})$.

\begin{lemma} \label{conj1}
In $TVP_n$ the following conjugation rules hold
$$
\lambda_{ij}^{\gamma_k} =  \left\{
\begin{array}{ll}
\vspace{0.1 cm}

\lambda_{ij} & \text{for }~ i<k<j \text{ or } k<i \text{ or } k>j,  \\
\vspace{0.1 cm}

\lambda_{ji}^{\gamma_j} & \text{ for } k=i, \\ 
\vspace{0.1 cm}

\lambda_{ji}^{\gamma_i} & \text{ for } k=j.
\end{array} \right.
$$
\end{lemma} 

Using this lemma and the presentation of $TVP_n$ it is easy to prove the following.

\begin{proposition}
The  kernel $PT_n = \ker(\varphi_{PT})$ is the group that is generated by the elements 
$$\lambda_{ij},~~\lambda^{\gamma_i}_{ij},~~\lambda^{\gamma_j}_{ij},~~1 \leq i\neq j \leq n,$$ 
with defining relations (\ref{comm-clas}), (\ref{classical}), relations from Lemma \ref{conj1}, and the following relations:
$$\lambda_{ij}^{m}\lambda_{kl}^{n}=\lambda_{kl}^{n}\lambda_{ij}^{m}, \text{ where } m\in \{e,\gamma_i,\gamma_j\}, \text{ and } n\in \{e, \gamma_k,\gamma_l\},$$

$$\lambda_{ki}\lambda_{kj}^{\gamma_j}\lambda_{ij}^{\gamma_j} =\lambda_{ij}^{\gamma_j}\lambda_{kj}^{\gamma_j}\lambda_{ki}\text{, \hspace{0.3cm}} \lambda_{ki}^{\gamma_k}\lambda_{kj}^{\gamma_k}\lambda_{ij}=\lambda_{ij}\lambda_{kj}^{\gamma_k}\lambda_{ki}^{\gamma_k},$$

$$\lambda_{ki}^{\gamma_i}\lambda_{kj}\lambda_{ij}^{\gamma_i} =\lambda_{ij}^{\gamma_i}\lambda_{kj}\lambda_{ki}^{\gamma_i}\text{, \hspace{0.3cm}} \lambda_{ik}\lambda_{kj}^{\gamma_k}\lambda_{ij}^{\gamma_i} =\lambda_{ij}^{\gamma_i}\lambda_{kj}^{\gamma_k}\lambda_{ik},$$

$$\lambda_{ki}^{\gamma_k}\lambda_{jk}\lambda_{ij}^{\gamma_j} =\lambda_{ij}^{\gamma_j}\lambda_{jk}\lambda_{ki}^{\gamma_k} \text{, \hspace{0.3cm}} \lambda_{ki}^{\gamma_i}\lambda_{kj}^{\gamma_j}\lambda_{ji}=\lambda_{ji}\lambda_{kj}^{\gamma_j}\lambda_{ki}^{\gamma_i},$$

$$\lambda_{ik}\lambda_{jk}\lambda_{ji} =\lambda_{ji}\lambda_{jk}\lambda_{ik}.$$
\end{proposition}

As corollary we get

\begin{corollary}
 $TVB_n = PT_n \rtimes TS_n$.
\end{corollary}

 Comparing the presentation of of $PT_n$ with the presentation of $PL_n$, we can see that these groups are isomorphic.

\begin{proposition} \label{isom1}
$PT_n \cong PL_n$ for all $n \geq 2$.
\end{proposition}

\begin{proof}
From Lemma 4.2, we can notice, for $i>j$; $\lambda_{ij}^{\gamma_i}=\lambda_{ji}^{\gamma_j}$ and $\lambda_{ij}^{\gamma_j}=\lambda_{ji}^{\gamma_i}$.

Also, $\lambda_{ij}^{\gamma_i}=\lambda_{ji}^{\gamma_j}$ implies $\lambda_{ij}=\lambda_{ji}^{\gamma_j\gamma_i}$.
So, we can remove elements $\lambda^{\gamma_k}_{ij}$ for $i > j$ and replace $\lambda_{ij}$ by $\lambda_{ji}^{\gamma_j\gamma_i}$ for $i > j$ .
Therefore, proved that the generators of the groups $PL_n$ and $PT_n$ are identical, it follows that the relations will also be the same.

\end{proof}

\medskip

We want to find the kernel $\ker(\varphi_{HT})$.

\begin{example}
Let us consider the case  $n=3$. In this case,  $\ker(\varphi_{HT})$  is  the normal closure of $B_3$ in $TVB_3$. As we know, the normal closure of $B_3$ in $VB_3$ is the subgroup $VH_3$ which is generated by elements
$$
x_{12} = \sigma_1,~~x_{13} = \rho_2 \sigma_1 \rho_2,~~x_{23} = \sigma_2,
$$
$$
x_{21} = \rho_1 \sigma_1 \rho_1,~~x_{31} = \rho_2 \rho_1 \sigma_1 \rho_1  \rho_2,~~x_{32} = \rho_2 \sigma_2 \rho_2.
$$
In   $TVB_3$ we have relations
$$
\gamma_1 \sigma_2 = \sigma_2 \gamma_1,~~\gamma_3 \sigma_1 = \sigma_1 \gamma_3, 
$$
$$
\rho_1 \sigma_1 \rho_1 = \gamma_2  \gamma_1  \sigma_1 \gamma_1 \gamma_2,~~\rho_2 \sigma_2 \rho_2 = \gamma_3  \gamma_2  \sigma_2 \gamma_2 \gamma_3.
$$
We can rewrite these relations in the form
$$
x_{23}^{\gamma_1} = x_{23},~~~x_{12}^{\gamma_3} = x_{12},
$$
$$
x_{21}^{ \gamma_2} = x_{12}^{\gamma_1},~~~x_{32}^{ \gamma_3} = x_{23}^{\gamma_2}.
$$
Conjugating these relations by $\rho_1$, we get
$$
x_{13}^{\gamma_2} = x_{13},~~~x_{21}^{\gamma_3} = x_{21},
$$
$$
x_{12}^{ \gamma_1} = x_{21}^{\gamma_2},~~~x_{31}^{ \gamma_3} = x_{13}^{\gamma_1}.
$$
Conjugating these relations by $\rho_2$, we get
$$
x_{32}^{\gamma_1} = x_{32},~~~x_{13}^{\gamma_2} = x_{13},
$$
$$
x_{31}^{ \gamma_3} = x_{13}^{\gamma_1},~~~x_{23}^{ \gamma_2} = x_{32}^{\gamma_3}.
$$
\end{example}

Analysing the conjugation rules from this example, one can prove the next lemma.

\begin{lemma} \label{conj}
In $TVH_n$ the following conjugation rules hold
$$
x_{ij}^{\gamma_k} =  \left\{
\begin{array}{ll}
\vspace{0.1 cm}

x_{ij} & \text{for }~ i<k<j \text{ or } k<i \text{ or } k>j,  \\ 

\vspace{0.1 cm}

x_{ji}^{\gamma_j} & \text{ for } k=i, \\ 
\vspace{0.1 cm}

x_{ji}^{\gamma_i} & \text{ for } k=j.
\end{array} \right.
$$
\end{lemma} 

Using this lemma and the presentation of $TVH_n$ it is easy to prove 


\begin{proposition}
The  kernel $HT_n=\ker(\varphi_{HT})$ is the group that is generated by the elements 
$$x_{ij},~~x^{\gamma_i}_{ij},~~ x^{\gamma_j}_{ij},~~1 \leq i\neq j \leq n,$$ 

with defining relations (\ref{comm-clas2}), (\ref{2classical2}), relations from Lemma \ref{conj}  and the following relations:
$$x_{ij}^{\gamma_m}x_{kl}^{\gamma_n} =x_{kl}^{\gamma_n}x_{ij}^{\gamma_m}, \text{ where } m\in \{i,j\}, \text{ and } n\in \{k,l\},$$

$$x_{ik}^{\gamma_i}x_{kj}x_{ik}^{\gamma_i} =x_{kj}x_{ik}^{\gamma_i}x_{kj}\text{ , }x_{ik}^{\gamma_k}x_{kj}^{\gamma_k}x_{ik}^{\gamma_k} =x_{kj}^{\gamma_k}x_{ik}^{\gamma_k}x_{kj}^{\gamma_k},$$

$$x_{ik}x_{kj}^{\gamma_j}x_{ik} =x_{kj}^{\gamma_j}x_{ik}x_{kj}^{\gamma_j} \text{ , }x_{ki}x_{kj}^{\gamma_k}x_{ki} =x_{kj}^{\gamma_k}x_{ki}x_{kj}^{\gamma_k},$$

$$x_{ik}^{\gamma_i}x_{kj}^{\gamma_j}x_{ik}^{\gamma_i} =x_{kj}^{\gamma_j}x_{ik}^{\gamma_i}x_{kj}^{\gamma_j} \text{ , }x_{ik}^{\gamma_k}x_{jk}x_{ik}^{\gamma_k} =x_{jk}x_{ik}^{\gamma_k}x_{jk},$$

$$x_{ki}x_{jk}x_{ki}=x_{jk}x_{ki}x_{jk}.$$
\end{proposition}

\begin{corollary}
$TVB_n = HT_n \rtimes TS_n$.
\end{corollary}

Likewise Proposition~\ref{isom1}, it can be proved the next proposition.

\begin{proposition} \label{isom2}
$HT_n \cong HL_n$ for all $n \geq 2$.
\end{proposition}

From results of this paper we get the follows proposition

\begin{proposition} 
1) In the next diagram rows and columns are short  exact sequences:
$$
\begin{matrix} 
 & & 1 & & 1 & & 1 &  \\
  &   & \downarrow &   & \downarrow &   & \downarrow &  &   \\ 
 1  &  \longrightarrow  & PL_n = \ker(\psi_P)&  \overset{\cong}{\longrightarrow}  & PT_n = \ker(\varphi_{PT}) &  \longrightarrow  &  1 &  \longrightarrow & 1  \\ 
  &   & \downarrow &   & \downarrow &   & \downarrow &  &   \\  
1  &  \longrightarrow  & TVP_n = \ker(\varphi_{P})  &  \longrightarrow  & TVB_n &  \overset{\varphi_{P}}{\longrightarrow} & S_n &  \longrightarrow & 1  \\ 
  &   & \downarrow {\text{\tiny $\psi_P$} } &   & \downarrow  {\text{\tiny $\varphi_{PT}$} } &   & \| &  &   \\ 
1  &  \longrightarrow  & A_n &  \longrightarrow  & TS_n &  \longrightarrow  & S_n &  \longrightarrow & 1  \\ 
  &   & \downarrow &   & \downarrow &   & \downarrow &  &   \\ 
 & & 1 & & 1 & & 1 &  \\
    \end{matrix}
$$

2) In the next diagram rows and columns are short  exact sequences:
$$
\begin{matrix} 
 & & 1 & & 1 & & 1 &  \\
  &   & \downarrow &   & \downarrow &   & \downarrow &  &   \\ 
1  &  \longrightarrow  &HL_n = \ker(\psi_T) & \overset{\cong}{\longrightarrow} & HT_n = \ker(\varphi_{HT}) &  \longrightarrow  &  1 &  \longrightarrow & 1  \\ 
  &   & \downarrow &   & \downarrow &   & \downarrow &  &   \\  
1  &  \longrightarrow  & TVH_n = \ker(\varphi_{T}) &  \longrightarrow  & TVB_n &  \overset{\varphi_{T}}{\longrightarrow}  & S_n &  \longrightarrow & 1  \\ 
  &   & \downarrow {\text{\tiny $\psi_T$} } &   & \downarrow {\text{\tiny $\varphi_{HT}$} } &   & \| &  &   \\ 
1  &  \longrightarrow  & A_n &  \longrightarrow  & TS_n &  \longrightarrow  & S_n &  \longrightarrow & 1  \\ 
  &   & \downarrow &   & \downarrow &   & \downarrow &  &   \\ 
 & & 1 & & 1 & & 1 &  \\
    \end{matrix}
$$
\end{proposition}

\bigskip

\section*{Concluding remarks}
In this paper, we study twisted virtual braid group and some of its subgroups. As a future work, it will be interesting to find an answers on the next questions.

1) There exists some representations of $VB_n$ by automorphisms of some groups \cite{BN, BMN, BKV}. Is it possible to extend these representations to representations of $TVB_n$?

2) M. O. Bourgoin \cite{Bour} introduced twisted knot theory as a generalization of the classical knot theory. Is it possible to use representations of $TVB_n$ construct group invariants for twisted virtual knots? For virtual knots, this has been done in the work \cite{BMN-1}.

3) The concept of a subgroup of  camomile-type were introduced in \cite{BK-1} (see also \cite{BW}). Is it true that $TVP_n$, $n\geq 4$ is a subgroup of camomile type in $TVB_n$ with the highlighted petal $TVP_4$?


\end{document}